\documentclass[a4paper,12pt]{article} 

\usepackage{cmap}					
\usepackage{mathtext} 				
\usepackage[T2A]{fontenc}			
\usepackage[utf8]{inputenc}			
\usepackage[english,russian]{babel}	
\usepackage{geometry}

\usepackage{amsmath,amsfonts,amssymb,amsthm,mathtools} 
\usepackage{icomma} 
\usepackage[all]{xy}
\mathtoolsset{showonlyrefs=true} 

\usepackage{euscript}	 
\usepackage{mathrsfs} 
\newtheorem*{corollary}{Corollary}
\newtheorem{theorem}{Theorem}

\newtheorem*{propA}{Proposition A}
\newtheorem*{propB}{Proposition B}

\DeclareMathOperator{\Real}{\mathop{Re}}

\usepackage{graphicx}  
\graphicspath{{images/}{images2/}}  
\usepackage{wrapfig} 

\usepackage{array,tabularx,tabulary,booktabs} 
\usepackage{longtable}  
\usepackage{multirow} 

\author{V.~Iudelevich }
\title{On the mean value of the functions related to the divisor function on the ring of polynomials over a finite field.}
\date{}

\begin{document} 
	
	\maketitle
	\selectlanguage{english}
	\begin{abstract}
	Let $ \mathbb{F}_q[T]$\, be the ring of polynomials over a finite field $ \mathbb{F} _q $. \\Let $ g: \mathbb{F}_q[T] \rightarrow \mathbb{R} $ be a multiplicative function such that for any irreducible polynomial $ P $ over $ \mathbb{F} _q $ and any $ k \ge 1 $, the equality $ d_k = g (P ^ k) $ holds for some arbitrary sequence of reals $\{d_k\}_{k=1}^{\infty}$. In this paper, we get an explicit formula for the sum $$ T (N) = \sum\limits_{\substack{\deg F=N \\  F \text{ is monic}}}{g (F)}, $$ and also derive different asymptotics when this sum in cases of \\ $ q \to \infty; \ q \to \infty, \ N \to \infty; \ q ^ N \to \infty $.
	\end{abstract}

	\section*{Introduction}
	Let $ q$~be a prime power. In what follows, $ \mathbb{F}_q $  denote the finite field of $q$ elements. Let $ \mathbb{F}_q [T] $ be the polynomial ring over $ \mathbb {F}_q $. It is well known that $ \mathbb{F}_q [T] $ is a Euclidean ring. In particular, there is the theorem of unique factorization in this ring, i.e.,
	any polynomial $ F \neq const$ can be represented in the form 	\begin{equation}F=a \cdot P_1^{e_1}P_2^{e_2}\ldots P_k^{e_k},
	\label{Pk}
	\end{equation}
	where $ a \in \mathbb{F}_q ^ * $, $P_1, P_2, \ldots, P_k $ are distinct irreducible polynomials over $ \mathbb {F} _q $ with unitary leading coefficients (monic polynomials), and $ e_1, e_2, \ldots, e_k $ are positive integers. This decomposition is unique up to order.
	The above theorem allows us to consider analogues to some well-known multiplicative functions. Recall that for a monic polynomial $F$, the divisor function $ \tau (F) $ is defined by 
	$$
	\tau(F)\,=\,\sum\limits_{D|F}{1}
	$$
	where the summation is taken over all monic divisors of $F$.
	In other words, $ \tau (F) $ is the number of solutions of the equation
	$F_1F_2=F$ in monic polynomials. 
	The generalized divisor function $ \tau_m (F) $, $ m \geq 2 $, is defined in a similar way as the number of solutions of the equation $ F_{1} F_{2} \ldots F_{m} = F.$
	\par Given a polynomial $ F $ of degree $ n $, its norm $ N(F) $ is defined by $ N (F) = q^n $.
	Clearly, for any polynomials $ F $ and $ G $ we have
	\[
	N(FG)=N(F)N(G).
	\]
	The most important object in studying the arithmetic properties of the ring $ \mathbb {F}_q [T]$ is its zeta function $ \zeta_q (s) $. 
	For $ s = \sigma + it, \sigma> 1 $, zeta function $ \zeta_q (s) $ is defined as
	\begin{equation}
	\zeta_q(s)=\sum\limits_{\substack{F\in\mathbb{F}_q[T]\\ F\,\text{is monic} }}{\dfrac{1}{N^s(F)}}.
	\label{eulersum}
	\end{equation}
	There are $q^n$ monic polynomials of degree $n$ in $\mathbb F_q[T]$, so we get
	\begin{equation}
	\zeta_q(s)=\sum\limits_{\substack{F\in\mathbb{F}_q[T]\\ F\,\text{is monic} }}{\dfrac{1}{N^s(F)}}=\sum_{n=0}^\infty{\dfrac{q^n}{q^{ns}}}=\dfrac{1}{1-q^{1-s}}.
	\label{analit}
	\end{equation}
	It follows that $ \zeta_q(s) $ can be  continued to a meromorphic function $ {(1-q ^ {1-s})} ^ {- 1} $ on the whole
	complex plane with simple poles at the points
	$$ s_k = 1 + \frac {2 \pi k} {\ln q} i, \ \ k \in \mathbb{Z}. $$
	A simple computation shows that the residue at the point $s_k$ is equal to $(\ln q)^{-1}.$
	\par The unique factorization theorem leads to the following identity:
	\begin{equation}
	\zeta_q(s)=\prod_{P\in\mathbb{F}_q[T]}{\left(1-N^{-s}(P)\right)}^{-1},
	\label{eulerprod}
	\end{equation}
	where $\Real s > 1$ and the product is taken over all monic irreducibles $P$ (Euler product).
	\par It is interesting to study the average values of multiplicative functions over the ring $ \mathbb F_q [T] $. For the first time, such problems were   considered by L.~Carlitz. In \cite{Iud_bib_1}, he obtained precise formulas for the average values of some multiplicative functions.
	\par The possibility of obtaining explicit (not asymptotic)
	formulas in problems of such type are explained by simple nature of $\zeta_q(s)$ and the fact that the corresponding generating Dirichlet series is presented in the form $ ( \zeta_q (n_1s)) ^ {m_1} \ldots (\zeta_q (n_ks)) ^ {m_k} $. 
	Therefore, in particular, the problem of determining the value
	\begin{equation}
	\sum_{\deg F=n}{\tau_m(F)}
	\label{Iud_f_1}
	\end{equation}
	reduces to calculating the coefficient for $ q^{- ns} $ in the series
	$$
	(\zeta_q(s))^m=\dfrac{1}{(1-q^{1-s})^m}.
	$$
	The value \eqref{Iud_f_1} is analogue of the sum
	\begin{equation}
	\sum_{n\leq x}{d_m(n)},
	\label{Iud_f_2}
	\end{equation}
	where $ d_m (n) $ equals to the number of solutions of the equation $ x_1x_2 \ldots x_m = n $ in positive integers $ x_1 , x_2, \ldots, x_n.$
	The study of \eqref{Iud_f_2} is the subject of the generalized Dirichlet divisor problem.
	Along with \eqref{Iud_f_2}, the sums $\sum\limits_{n\leq x}{\dfrac{1}{d_m(n)}}$
	have been studied; we have the asymptotic formula (see  \cite{Iud_bib_2})
	$$
	\sum\limits_{n\leq x}{\dfrac{1}{d_m(n)}}=\dfrac{x}{(\ln{x})^{1-\frac{1}{m}}}\left(a_0+\dfrac{a_1}{\ln{x}}+\ldots+\dfrac{a_N}{(\ln{x})^N}+O_N\left(\dfrac{1}{(\ln{x})^{N+1}}\right)\right)
	$$
	as $ x \to + \infty, $ $ N \geq 0 $~is an arbitrary fixed number, $ a_0, a_1, \ldots, a_N, \ldots $~are some constants.
	\par Let us fix a sequence of real numbers $\left\{d_k\right\}_{k=1}^{\infty}$. We will assume that $ g: \mathbb{F} _q[T] \rightarrow \mathbb{R} $ is a multiplicative function such that for any irreducible polynomial $ P $ and integer $ k \ge 1 $ the equality 
	\begin{equation}
	d_k = g (P ^ k)
	\label {d_k}
	\end{equation}
	holds.
	Suppose that the power series
	\begin{equation}
	f (t) = f_g(t) = 1 + d_1t + d_2t ^ 2 + \ldots + d_k t ^ k + \ldots
	\label{f}
	\end{equation}
	converges in some circle centred at the origin.
	We define the sequence $ \{h_k\} _ {k = 1}^\infty $ by
		\begin{equation}
		f (t)(1-t)^{h_1}(1-t^2)^{h_2}\ldots(1-t^k)^{h_k} = 1 + h_{k + 1} t^{k +1} + \ldots.
		\label{hkk}
		\end{equation}
		It is easy to see that $ h_1 = d_1 $. Next, we define the sequence $ \{a_k\}_{k = 1}^\infty $ from the expansion
			\begin{equation}
			\ln f (t) = \sum_{k = 1}^\infty {a_k t^k}.
			\label{akk}
			\end{equation}
			In this paper, we study the asymptotic of the sum
			$$ T (N) = T_g (N) = \sum_ {\deg F = N} g (F), $$
			where $ F $ runs through monic polynomials of degree $ N $. We prove the following two theorems.
			\begin{theorem} \label{thm1}
		We have
			$$ T (N) = A_0 (N) q ^ N + A_1 (N) q ^ {N-1} + \ldots + A_ {N-1} (N) q, $$
			\begin{equation}
			A_l (N) = \sum \limits_{\substack{k_1 + 2k_2 + \ldots + (l + 1) k_{l + 1} = N \\ k_2 + 2k_3 + \ldots + lk_{l + 1} = l}} \binom{-h_1}{k_1} \binom{-h_2}{k_2}\ldots\binom {-h_{l + 1}} {k_{l + 1}}(-1)^{k_1 + k_2 + \ldots + k_{l + 1}}
			\label{Al}
			\end{equation}
			where the quantities $ h_i $ defined in \eqref{hkk}.
			 In particular, for any fixed $ N $ and $ q \to \infty $ we have
			$$ T (N) = A_0 (N) q ^ N + A_1 (N) q^{N-1} + \ldots + A_{n-1} (N) q^{N-n + 1} + O_N(q ^ {N-n}), $$
			where $ 1 \leq n \leq N. $
			\end{theorem}
			
			\begin{theorem} \label{thm2}
			Let	$ N\ge 190,\ \ 0 <d_1 <1, $ and for all $k\ge 1$ we have
			\begin{equation}
			 k | a_k | \le 1,
			\label{conda}
			\end{equation}
			where the values $ d_k $ and $ a_k $ are defined in \eqref{d_k} and \eqref{akk} respectively.
			\\ Then for $h\in \left[1,\frac{N}{36\ln N}\right]$ and $q\ge (17h)^{12h+9}$ we have
		\begin{equation}
		T(N)=A_0(N)q^N+A_1(N)q^{N-1}+\ldots+A_{h-1}(N)q^{N-h+1}+O\left(d_1 p(h)\dfrac{q^{N-h}}{N^{1-d_1}}\right),
		\label{TN}
		\end{equation}
			where the quantities $A_l (N), \ 0 \le l \le h-1 $ are defined in \eqref{Al},
			$$ | A_l (N) | \le \dfrac{3d_1p (l)}{N ^ {1-d_1}}, $$
			$ p(l) $ is number of partitions of $ l $, and implied constant in \eqref{TN} is (at most) 3.1. \\In particular, the equality \eqref{TN} holds for $q\to\infty,\,N\to\infty$ and fixed $h\ge 1$, also equality \eqref{TN} holds for $q\to\infty$, fixed $N\ge 190$ and $h\in \left[1,\frac{N}{36\ln N}\right]$. 
			\end{theorem}
			\par Next, we use the main technical result from the paper \cite{Iud_bib_3} to find the asymptotics $ T (N) $ in the case $ q^N\to\infty $.
			\par We conclude our work by giving some examples of the functions $ g $ and the expansions for the sums $ T_g(N) $.
		\section*{Proof of Theorem \ref{thm1}}
	\par Set $$\Phi(s)=\sum\limits_{\substack{F\in\mathbb{F}_q[T]\\ F \text{ is monic}}}{\dfrac{g(F)}{{N(F)}^s}}=\sum_{N=0}^{+\infty}{T(N)q^{-Ns}}.$$ Since $g$ is multiplicative, we have
	\begin{multline}\Phi(s)=\prod_{\substack{P\  \text{is monic}\\ \text{irreducible}}}{\left(1+\dfrac{g(P)}{N(P)^s}+\dfrac{g(P^2)}{N(P^2)^{s}}+\ldots \right)}=\\ \prod_{l=1}^{+\infty}{\left(1+\dfrac{d_1}{q^{ls}}+\dfrac{d_2}{q^{2ls}}+\ldots\right)^{\pi_q(l)}}=
	\prod_{l=1}^{+\infty}{(f(q^{-ls}))^{\pi_q(l)}},
	\end{multline}
	where $\pi_q(l)$ is the number of irreducible monic polynomial of degree $l$.
	Now we set \begin{equation}
	f_{k+1}(t)=f(t)(1-t)^{h_1}(1-t^2)^{h_2}\ldots(1-t^k)^{h_k}.
	\label{hk}
	\end{equation}
	Then for $k=n$ and $t=q^{-ls}$ we have
	$$f(q^{-ls})=f_{n+1}(q^{-ls})(1-q^{-ls})^{-h_1}\ldots(1-q^{-nls})^{-h_n}.$$
	Thus we get
	$$\Phi(s)=(\zeta_q(s))^{h_1}(\zeta_q(2s))^{h_2}\ldots(\zeta_q(ns))^{h_n}\prod_{l=1}^{+\infty}{(f_{n+1}(q^{-ls}))^{\pi_q(l)}}.$$
	Now we set $z=q^{1-s}$, $T_0(N)=q^{-N}T(N)$. Setting via
	$$\xi(z)=\sum_{N=0}^{+\infty}{T_0(N)z^N},$$
	we have
	\begin{equation} \xi(z)=\left(1-z\right)^{-h_1}\left(1-\dfrac{z^2}{q}\right)^{-h_2}\ldots\left(1-\dfrac{z^n}{q^{n-1}}\right)^{-h_n}\prod_{l=1}^{+\infty}{\left(f_{n+1}\left(\dfrac{z^l}{q^l}\right)\right)^{\pi_q(l)}}.
	\label{xi}
	\end{equation}
	Set
	$$\Pi(z)=\prod_{l=1}^{+\infty}{\left(f_{n+1}\left(\dfrac{z^l}{q^l}\right)\right)^{\pi_q(l)}}=\exp\Big\{\sum_{l=1}^{+\infty}\pi_q(l)\ln f_{n+1}\left(\dfrac{z^l}{q^l}\right)\Big \}.$$
	We define the sequnce $c_\nu$ from the expansion
	\begin{equation}
	\ln f_{n+1}(t)=\sum_{\nu=n+1}^{+\infty}c_{\nu}t^\nu.
	\label{cnu}
	\end{equation}
	Using the following equality
	$$\ln f_{n+1}(t)=\ln f(t)+h_1\ln (1-t)+\ldots+h_n\ln(1-t^n)$$
	and comparing coefficients at $t^\nu,\ \nu\ge n+1$ we have
	$$c_\nu=a_\nu-\sum\limits_{\substack{d|\nu \\
			d\le n}}{\dfrac{dh_d}{\nu}}.$$
	Hence 
	\begin{multline}
	\ln\Pi(z)=\sum_{l=1}^{+\infty}\pi_q(l)\ln f_{n+1}\left(\dfrac{z^l}{q^l}\right)=\sum_{l=1}^{+\infty}{\dfrac{1}{l}\sum_{d|l}{\mu(d)q^{\frac{l}{d}}\sum_{\nu\ge n+1}{c_\nu\dfrac{z^{\nu l}}{q^{\nu l}}}}}=\\
	\sum_{d=1}^{+\infty}{\dfrac{\mu(d)}{d}\sum_{\delta=1}^{+\infty}{\dfrac{q^\delta}{\delta}}\sum_{\nu\ge n+1}c_\nu\dfrac{z^{\nu d\delta}}{q^{\nu d \delta}}}=\sum_{k\ge n+1}{\dfrac{1}{q^k}\Big\{ \sum\limits_{\substack{\nu d \delta =k 
				\\\nu\ge n+1}}}{\dfrac{\mu(d)q^\delta c_\nu}{d\delta}}
	\Big\}z^k=\sum_{k\ge n+1}{A_k z^k},
	\end{multline}
	where
	\begin{equation}
	A_k={\dfrac{1}{q^k}\Big\{ \sum\limits_{\substack{\nu d \delta =k 
				\\\nu\ge n+1}}}{\dfrac{\mu(d)q^\delta c_\nu}{d\delta}}
	\Big\}.
	\label{Ak}
	\end{equation}
	Thus we have  
	$$\prod_{l=1}^{+\infty}{\left(f_{n+1}\left(\dfrac{z^l}{q^l}\right)\right)^{\pi_q(l)}}=\exp\Big\{\sum_{k\ge n+1}{A_k z^k}\Big \}=\sum_{k=0}^{+\infty}{B_k}z^k$$
	for some sequence $B_k$ suth that $B_0=1$  and $B_1=B_2=\ldots=B_n=0$.
	Hence
	\begin{multline} \xi(z)=\left(1-z\right)^{-h_1}\left(1-\dfrac{z^2}{q}\right)^{-h_2}\ldots\left(1-\dfrac{z^n}{q^{n-1}}\right)^{-h_n}\sum_{k=0}^{+\infty}{B_k z^k}=\\
	\left( \sum\limits_{k_1 = 0}^{+\infty}  (-1)^{k_1} \binom{-h_1}{k_1} z^{k_1} \right)  \left( \sum\limits_{k_2 = 0}^{+\infty}  (-1)^{k_2} \binom{-h_2}{k_2} \dfrac{z^{2 k_2}}{q^{k_2}} \right) \ldots\\ \ldots\left( \sum\limits_{k_n = 0}^{+\infty}(-1)^{k_n} \binom{-h_n}{k_n} \dfrac{z^{nk_n}}{q^{(n-1)k_n}} \right)  \sum\limits_{k=0}^{+\infty} B_k z^k.
	\end{multline}
	By comparing the coefficient of $z^N$ we get
	$$T_0(N)=\sum\limits_{\substack{ k+k_1+2k_2+\ldots+nk_n=N\\ k,\,k_i\ge 0\\}} {\binom{-h_1}{k_1}\ldots \binom{-h_n}{k_n}\dfrac{(-1)^{k_1 + \ldots + k_n}B_k}{q^{k_2 + 2k_3 + \ldots + (n-1)k_n}}}.$$
	We divide this sum into three parts
	\begin{equation}
	T_0(N)=S_1+S_2+S_3.
	\label{Si}
	\end{equation}
	Here $ S_1 $ involves the terms  with $ k = 0 $ and
	$$ k_2 + 2k_3 + \ldots + (n-1) k_n \leq n-1; $$
	$ S_2 $ contains the terms from $S$ with $ k = 0 $ and
	$$ k_2 + 2k_3 + \ldots + (n-1) k_n \geq n. $$
	Finally, $S_3$ includes those terms with $k\geq 1$.
	Since $ B_k = 0 $ for $ k = 1, 2, \ldots, n $,  the components of $ S $ with $ k \geq n + 1 $ are included in $ S_3 $.
Note that
\begin{multline}S_1=\sum\limits_{\substack{k_1+2k_2+\ldots+nk_n=N\\ k_2+2k_3+\ldots+(n-1)k_n\le n-1}}{\binom{-h_1}{k_1}\ldots \binom{-h_n}{k_n}\dfrac{(-1)^{k_1 + \ldots + k_n}}{q^{k_2 + 2k_3 + \ldots + (n-1)k_n}}}=\\
\sum_{l=0}^{n-1}{{q^{-l}}\sum\limits_{\substack{ k_1+2k_2+\ldots+(l+1)k_{l+1}=N\\ k_2+2k_3+\ldots+lk_{l+1}=l}}}\binom{-h_1}{k_1}\ldots \binom{-h_{l+1}}{k_{l+1}} (-1)^{k_1 + k_2 + \ldots + k_{l+1}}=\sum_{l=0}^{n-1}A_l(N)q^{-l}.
\end{multline}
	If we choose  $n=N$, then the sums $S_2$ and $S_3$ will be empty. Then we have
	$$T_0(N)=S_1=\sum_{l=0}^{N-1}A_l(N)q^{-l},$$
The proof is completed by multiplying both sides  by $ q ^ N $. 
\section*{Proof of Theorem 2}
	Taking the logarithm from both sides of the equation \eqref{hk}, expanding both sides in power series of $t$, and by comparing the coefficient of $t^k$, we get
	$$a_k=\sum_{d|k}{\dfrac{dh_d}{k}}.$$
    M\"obius inversion formula implies that
	$$h_k=\sum_{d|k}{\dfrac{\mu(d)a_{k/d}}{d}}.$$
In view of \eqref{conda}, for all $k\ge 1$ we have
	\[
	|h_k|=\Big|\sum_{d|k}{\dfrac{\mu(d)a_{k/d}}{d}}\Big|\le\dfrac{\tau(k)}{k}\le 1.
	\]
	Further, by \eqref{conda} and the above estimate we have
	\[
	|c_\nu|=\Big |a_\nu-\sum\limits_{\substack{d|\nu \\d\le n}}{\dfrac{dh_d}{\nu}}\Big |\le\dfrac{1}{\nu}\left({1+\sum_{d\le n}\tau(d)}\right)\le\dfrac{n\ln n+n+1}{\nu}=\dfrac{b_1(n)}{\nu}.
	\]
	Next,
	$$|A_k|=\dfrac{1}{q^k}\Big|\sum\limits_{\substack{d\delta\nu=k \\ \nu\ge n+1}}{\dfrac{q^\delta\mu(d)c_\nu}{d\delta}}\Big|\le\dfrac{b_1(n)}{kq^k}\sum\limits_{\substack{\nu|k \\ \nu\ge n+1}}\sum_{d|\frac{k}{\nu}}q^{\frac{k}{\nu d}}\le\dfrac{4b_1(n)}{kq^{\frac{n}{n+1}}}.$$
	Now we estimate the values of $B_k$ defined by expansion
	$$\exp\Big\{\sum_{k\ge n+1}{A_k z^k}\Big \}=\sum_{k=0}^{+\infty}{B_k}z^k.$$
	We have
	$$B_k=\sum_{r\ge 1}\dfrac{1}{r!}\sum\limits_{\substack{i_1,\ldots,i_r\ge n+1 \\ i_1+\ldots+i_r=k}}A_{i_1}\ldots A_{i_r}. $$
	Let $$I(k,r)=\#\{(i_1,\ldots,i_r)\in\mathbb{Z}^r: i_m\ge n+1,\ i_1+\ldots+i_r=k \},$$
	then $$I(k,r)=\binom{k-nr-1}{r-1}<\dfrac{k^r}{(r-1)!},\ \ 1\le r \le \dfrac{k}{n+1}.$$
	Further, under the condition $i_m\ge n+1$ and $i_1+\ldots+i_r=k$ the following inequality holds
	$$|A_{i_1}A_{i_2}\ldots A_{i_r}|\le\left(\dfrac{4b_1(n)}{n+1}\right)^r\dfrac{1}{q^{\frac{n}{n+1}k}}=\dfrac{(b(n))^r}{q^{\frac{n}{n+1}k}},$$
	where
	$$b(n)=\dfrac{4(n\ln n+n+1)}{n+1}<4(\ln n+1).$$
	Hence
	\begin{equation}
	|B_k|\le\sum_{1\le r \le\frac{k}{n+1}}\dfrac{I(k,r)}{r!}\max\limits_{\substack{i_m\ge n+1 \\ i_1+\ldots+i_r=k}}|A_{i_1}A_{i_2}\ldots A_{i_r}|\le\sum_{r=1}^{+\infty}\dfrac{(kb(n))^r}{r!(r-1)!}\dfrac{1}{q^{\frac{n}{n+1}k}}<\dfrac{e^{kb(n)}}{q^{\frac{n}{n+1}k}}.
	\end{equation}
	Finally, we get
	\begin{equation}
	|B_k|\le\dfrac{e^{kb(n)}}{q^{\frac{n}{n+1}k}}.
	\label{Bk}
	\end{equation}
	We recall that
	$$T_0(N)=\sum\limits_{\substack{ k+k_1+2k_2+\ldots+nk_n=N\\ k,\,k_i\ge 0\\}} {\binom{-h_1}{k_1}\ldots \binom{-h_n}{k_n}\dfrac{(-1)^{k_1 + \ldots + k_n}B_k}{q^{k_2 + 2k_3 + \ldots + (n-1)k_n}}}$$
	and
	$$T_0(N)=S_1+S_2+S_3,$$
	the values $S_i$ was defined on page \pageref{Si}. It turns out that the first term will make up the leading term in the asymptotics, and the sums $ S_2 $ and $ S_3 $ will give the remainder term.
	\par Take $q\ge (en)^{2n+9},\ 6\le n\le\frac{N}{6\ln N},\,n=6h, h\ge 1$.
	\par We consider the sum $S_1$ firstly. We have
	$$S_1=\sum_{l=0}^{n-1}A_l(N)q^{-l},$$
	where
	$$A_l(N)=\sum\limits_{\substack{ k_1+2k_2+\ldots+(l+1)k_{l+1}=N\\ k_2+2k_3+\ldots+lk_{l+1}=l}}\binom{-h_1}{k_1}  \binom{-h_2}{k_2} \ldots \binom{-h_{l+1}}{k_{l+1}} (-1)^{k_1 + k_2 + \ldots + k_{l+1}}.$$
	For all $m\ge 1$
	$$0<(-1)^m\binom{-h_1}{m}=(-1)^m\binom{-d_1}{m}=\dfrac{d_1}{m}\prod_{r=1}^{m-1}\left(1+\dfrac{d_1}{r}\right)\le\dfrac{d_1}{m}\exp\Big\{d_1\sum_{r=1}^{m-1}\dfrac{1}{r} \Big \}\le\dfrac{d_1e^{d_1}}{m^{1-d_1}}.$$
Since $d_1\in (0,1)$, the following estimate holds
	\begin{equation}
	(-1)^m\binom{-h_1}{m}\le d_1.
	\label{d1}
	\end{equation}
Further, by the inequality $|h_i| \le 1$, $i = 1, 2, \ldots$, we have
	$$ \left| \binom{-h_i}{k_i} \right|\le 1.$$
	Let us estimate the values of $A_l(N)$. For $l\ge 1$ define the value $\theta$ from the equality
	$$2k_2+3k_3+\ldots+(l+1)k_{l+1}=\theta l.$$
	Then $1 \le \theta\le 2$, hence
	$$k_1=N-2k_2-\ldots-(l+1)k_{l+1}=N-\theta l> N-2n.$$  Since
	 $n\le\frac{N}{6\ln N}$ and $N\ge 190$, then
	\begin{multline*}
	0<(-1)^{k_1}\binom{-h_1}{k_1}\le\dfrac{d_1e^{d_1}}{(N-2n)^{1-d_1}}=\dfrac{d_1e^{d_1}}{N^{1-d_1}}\dfrac{1}{(1-\frac{2n}{N})^{1-d_1}}<\\
	\dfrac{d_1e^{d_1}}{N^{1-d_1}}\dfrac{1}{(1-\frac{1}{3\ln N})^{1-d_1}}\le\dfrac{d_1e^{d_1}}{N^{1-d_1}}\dfrac{1}{(1-\frac{1}{3\ln 190})^{1-d_1}}<\dfrac{3d_1}{N^{1-d_1}}.
	\end{multline*}
	Since the sum of $A_l(N)$ contains exactly $p(l)$ terms (where $p(l)$ denote the number of partitions of $l$), we get
	$$|A_l(N)|\le\dfrac{3d_1p(l)}{N^{1-d_1}}.$$
	This estimate also holds for $l=0$ since
	$$A_0(N)=(-1)^N\binom{-h_1}{N}\le\dfrac{d_1 e^{d_1}}{N^{1-d_1}}<\dfrac{3 d_1 p(0)}{N^{1-d_1}}.$$
	Now we consider $S_2$. We have
	\begin{multline}
	S_2=\sum\limits_{\substack{ k_1+2k_2+\ldots+nk_{n}=N\\ k_2+2k_3+\ldots+(n-1)k_{n}\ge n}}{\binom{-h_1}{k_1}\ldots \binom{-h_n}{k_n}\dfrac{(-1)^{k_1 + \ldots + k_n}}{q^{k_2 + 2k_3 + \ldots + (n-1)k_n}}}=\\
	=\sum_{n\le l\le N}q^{-l}\sum\limits_{\substack{ k_1+2k_2+\ldots+nk_{n}=N\\ k_2+2k_3+\ldots+(n-1)k_{n}=l}}\binom{-h_1}{k_1}\ldots \binom{-h_n}{k_n}(-1)^{k_1 + \ldots + k_n}=\\
	=\left(\sum_{n\le l\le \frac{N}{3}}+\sum_{\frac{N}{3}< l\le N}\right)q^{-l}\sum\limits_{\substack{ k_1+2k_2+\ldots+nk_{n}=N\\ k_2+2k_3+\ldots+(n-1)k_{n}=l}}\binom{-h_1}{k_1}\ldots \binom{-h_n}{k_n}(-1)^{k_1 + \ldots + k_n}=W_1+W_2.
	\end{multline}
	In the sum $W_1$ we define the value $\theta_1$ from the equality
	$$2k_2+\ldots+nk_n=\theta_1 l.$$
	Then $1\le\theta_1\le 2$ and $k_1=N-\theta_1 l\ge N/3$, from here we get
	$$0<(-1)^{k_1}\binom{-h_1}{k_1}\le\dfrac{d_1e^{d_1}3^{1-d_1}}{N^{1-d_1}}\le\dfrac{3d_1}{N^{1-d_1}}.$$
	The remaining factors $| \binom{-h_i}{k_i}|$ in the sum of $W_1$ will be evaluated by one.
	Then since the number of solutions of the equation $k_2+2k_3+\ldots+(n-1)k_n=l$ with unknown $k_i$ does not exceed $ \dfrac{l^{n-1}}{(n-1)!}$, we have
	\[
	|W_1|\le\dfrac{3d_1}{N^{1-d_1}}\sum_{n\le l \le \frac{N}{3}}\dfrac{l^{n-1}}{(n-1)!q^l}=\dfrac{3d_1}{N^{1-d_1}(n-1)!q^n}\sum_{l=0}^{\frac{N}{3}-n}\dfrac{(l+n)^{n-1}}{q^l}.
	\]
	Since the inequality $(l+n)^{n-1}\le q^{l/2}$ holds for $q\ge(n+1)^{2n-2}$, then
	\begin{multline}
	|W_1|\le\dfrac{3d_1}{N^{1-d_1}(n-1)!q^n}\left(n^{n-1}+\sum_{l=1}^{+\infty}q^{-l/2}\right)=\\ \dfrac{3d_1}{N^{1-d_1}(n-1)!q^n}\left(n^{n-1}+\dfrac{1}{\sqrt q-1}\right)<\dfrac{3.1d_1n^n}{N^{1-d_1}q^n n!}.
	\end{multline}
	In the sum $W_2$ the facrors $|\binom{-h_i}{k_i}|, i\neq 1$ we estimate by one, and for the value $(-1)^{k_1}\binom{h_1}{k_1}$ we use the estimate \eqref{d1}. Then
	$$|W_2|\le \dfrac{d_1}{q^{N/3}}\sum_{N/3<n\le N}\dfrac{l^{n-1}}{(n-1)!}\le\dfrac{d_1N^n}{(n-1)!q^{N/3}}.$$
	Let us prove auxiliary inequality
	$$\dfrac{3.1d_1 n^n}{n!N^{1-d_1}q^n}>\dfrac{3\cdot6^5 d_1 N^n}{(n-1)!q^{N/3}}.$$
	For this purpose, it is sufficient to show that 
	$n<\dfrac{\frac{N \ln q}{3 \ln N}-1}{1+\frac{ \ln q}{ \ln N}}$. For $N\ge 190$ we have
	\[\dfrac{\frac{N \ln q}{3 \ln N}-1}{1+\frac{ \ln q}{ \ln N}}\ge\dfrac{\frac{N \ln 2}{3 \ln N}-1}{1+\frac{ \ln 2}{ \ln N}}\ge\left(1+\frac{\ln 2}{\ln 190}\right)^{-1}\dfrac{\ln 2}{3}\dfrac{N}{\ln N}-\left(1+\frac{\ln 2}{\ln 190}\right)^{-1}>\dfrac{N}{6\ln N}\ge n.
	\]
	Hence
	\[
	|S_2|\le|W_1|+|W_2|\le\dfrac{3.1d_1(1+\frac{1}{3\cdot 6^5})n^n}{n!N^{1-d_1}q^n}<\dfrac{3.2d_1n^n}{n!N^{1-d_1}q^n}.
	\]
	Finally, consider the sum $S_3$. We have
	\begin{multline}
	|S_3|=\left|\sum\limits_{\substack{ k+k_1+2k_2+\ldots+nk_n=N\\ k\ge n+1}}{\binom{-h_1}{k_1}\ldots \binom{-h_n}{k_n}\dfrac{(-1)^{k_1 + \ldots + k_n}B_k}{q^{k_2 + 2k_3 + \ldots + (n-1)k_n}}}\right|\le \\
	\sum_{k_1=0}^{N-n-1}(-1)^{k_1}\binom{-h_1}{k_1}\sum\limits_{\substack{k+2k_2+\ldots+nk_n=N-k_1 \\ k\ge n+1}}\dfrac{|B_k|}{q^{k_2+2k_3+\ldots+(n-1)k_n}}.
	\end{multline}
	Note that
	$$k_2+2k_3+\ldots+(n-1)k_n\ge\dfrac{N-k-k_1}{2}.$$
	Hence by virtue of \eqref{Bk} we get
	$$|S_3|\le\sum_{k_1=0}^{N-7}(-1)^{k_1}\binom{-h_1}{k_1}\sum\limits_{\substack{ k+2k_2+\ldots+nk_n=N-k_1\\ k\ge n+1}}\dfrac{1}{q^{\frac{N-k_1}{2}+k\left(\frac{1}{2}-\frac{1}{n+1}-\frac{b(n)}{\ln q}\right)}}.$$
	Since
\[\dfrac{1}{2}-\dfrac{1}{n+1}-\dfrac{b(n)}{\ln q}>\dfrac{1}{2}-\dfrac{1}{n+1}-\dfrac{4(\ln n+1)}{(2n+9)\ln(en)}
>\dfrac{1}{2}-\dfrac{1}{7}-\dfrac{4}{21}=\dfrac{1}{6}, \]
and $n=6h$ we will have
	$$k\left(\dfrac{1}{2}-\dfrac{1}{n+1}-\dfrac{b(n)}{\ln q}\right)>\dfrac{n+1}{6}>h.$$
Since the number of solutions of the equation $k+2k_2+\ldots+nk_n=N-k_1$ with unknown $k, k_2, \ldots, k_n$ does not exceed $\frac{(N-k_1)^n}{n!}$, then
	$$|S_3|\le\left(\sum_{0\le k_1\le N/2}+\sum_{N/2< k_1\le N-7}\right)(-1)^{k_1}\binom{-h_1}{k_1}\dfrac{(N-k_1)^n}{n!q^{\frac{N-k_1}{2}+h}}=W_3+W_4.$$
	Using the estimate \eqref{d1} we get
	$$|W_3|\le\dfrac{d_1N^{n+1}}{n!q^{\frac{N}{4}+h}}.$$
	In the sum $W_4$
	$$0<(-1)^{k_1}\binom{h_1}{k_1}\le\dfrac{d_1e^{d_1}2^{1-d_1}}{N^{1-d_1}}\le\dfrac{d_1e}{N^{1-d_1}}.$$
Hence
	$$|W_4|\le\dfrac{d_1e}{N^{1-d_1}n!q^h}\sum_{N/2<k_1\le N-7}\dfrac{(N-k_1)^n}{q^{\frac{N-k_1}{2}}}=\dfrac{d_1 e}{N^{1-d_1}n!q^h}\sum_{7\le l < N/2}\dfrac{l^n}{q^{l/2}}.$$ 
	Since   $q\ge(en)^{2n+9}>e^{\frac{4n\ln 7}{7}}$   for $n\ge 6$, then  we have $l^n\le q^{l/4}$ for $l\ge 7$. Hence
	$$|W_4|\le\dfrac{d_1 e}{N^{1-d_1}n!q^h}\sum_{l=7}^{+\infty}\dfrac{1}{q^{l/4}}=\dfrac{d_1 e}{N^{1-d_1}n!q^h(q^{7/4}-q^{3/2})}.$$
	If $q\ge (en)^{2n+9}$ and $N\ge 190$ then we have
	$$\dfrac{N^{n+1}}{n!q^{N/4+h}}<\dfrac{e}{N^{1-d_1}n!q^h(q^{7/4}-q^{3/2})},$$
then
	$$|S_3|\le|W_3|+|W_4|\le\dfrac{2ed_1}{N^{1-d_1}n!q^h(q^{7/4}-q^{3/2})}.$$
	Hence taking into proven estimates
	$$|S_2|+|S_3|\le\dfrac{3.2d_1n^n}{n!N^{1-d_1}q^n}+\dfrac{2ed_1}{n!N^{1-d_1}q^h(q^{7/4}-q^{3/2})}<\dfrac{4ed_1}{n!N^{1-d_1}(q^{7/4}-q^{3/2})q^h}.$$
	So, for $q\ge (en)^{2n+9},\ 6\le n\le\frac{N}{6\ln N},$ and $n=6h,\,h\ge 1$ we have
	$$T_0(N)=A_0(N)+A_1(N)q^{-1}+\ldots+A_{n-1}(N)q^{-n+1}+\Theta_1(h)\dfrac{1}{N^{1-d_1}q^h},$$
	where $$\left|\Theta_1(h)\right|\le\dfrac{4ed_1}{(6h)!(q^{7/4}-q^{3/2})}.$$
	Now
	\begin{multline}
	\left|A_h(N)q^{-h}+A_{h+1}(N)q^{-h-1}+\ldots+A_{6h-1}(N)q^{-6h+1}+\dfrac{\Theta_1(h)}{N^{1-d_1}q^h}\right|\le\\
	\dfrac{d_1q^{-h}}{N^{1-d_1}}\left(3p(h)+{3p(n)}\sum_{l=1}^{+\infty}\dfrac{1}{q^l}+\dfrac{12}{6!(q^{7/4}-q^{3/2})}\right)\le\dfrac{3.1d_1p(h)}{q^h N^{1-d_1}}
	\end{multline}
	for $q\ge (en)^{2n+9}$.
	Finally, we proved that
	$$T_0(N)=A_0(N)+A_1(N)q^{-1}+\ldots+A_{h-1}(N)q^{-h+1}+\dfrac{\Theta(h)}{N^{1-d_1}q^h},\ \left|\Theta(h)\right|\le 3.1d_1p(h)$$
	 for $q\ge (6eh)^{12h+9}$ and $1\le h\le \frac{N}{36\ln N}$.
	We complete our proof by multiplying both sides by $ q ^ N $.
	\section*{Corollary of Gorodetsky's Theorem}
	\par In the paper \cite{Iud_bib_3} Gorodetsky proves the following theorem
	\begin{theorem}\label{thGor}
		Let $$a(x)=\exp\left\{\sum\limits_{k=1}^\infty\dfrac{\widetilde{a}_k}{k}x^k\right\}=\sum\limits_{k=0}^\infty{a_k x^k},$$ and
		$$b(x)=\left(1-\dfrac{x}{\beta}\right)^{-c_1}=\sum\limits_{k=0}^\infty{b_k x^k},\  c_1\in(0,1)$$  be a two power series with radii of convergence at least $\alpha$ and exactly $\beta$, respectively. Assume that
		$\alpha>\beta>0$ and $r=\frac{\beta}{\alpha}\leq\frac{1}{\sqrt{2}}$, and for some positive costant $c_2$ the inequality $|\widetilde{a}_k|\leq{c_2}{\alpha^{-k}}$ holds.
		Then for all fixed $n\ge 0$ and $N>n$ for the coefficient 
		for $x^N$ in the product of $a(x)b(x)$ we have
		$$[x^N]a(x)b(x)=b_N\left(a(\beta)+\sum_{k=1}^n\dfrac{\binom{k-c_1}{k}}{\binom{N+c_1-1}{k}}\dfrac{\beta^k}{k!}a^{(k)}(\beta)+E\right),$$
		where $E\ll_{n,c_1,c_2}\left(\dfrac{r}{N}\right)^{n+1}.$
	\end{theorem}
	We apply this theorem to a function $g(F)$, for which inequality \eqref{conda} holds.  Put $x=zq^{-1}=q^{-s}, n=1,$ in the formula \eqref{xi},  then
	\begin{multline}\sum_{N=0}^\infty{T(N)x^N}=\left(1-\dfrac{x}{q^{-1}}\right)^{-d_1}\exp\left\{\sum_{l=1}^\infty\pi_q(l)\ln f_2(x^l) \right\}=\\
	\left(1-\dfrac{x}{q^{-1}}\right)^{-d_1}\exp\left\{\sum_{k=2}^\infty A_k q^k x^k \right\}=\left(1-\dfrac{x}{q^{-1}}\right)^{-d_1}a(x).
	\end{multline}
	Set $ \beta=q^{-1}, c_1=d_1.$
	Let $R$ be radius of convergence of the power series $\sum\limits_{k=2}^\infty A_k q^k x^k$. Then by virtue of the estimate \eqref{Ak} we get
	$$R^{-1}=\limsup\limits_{k\rightarrow\infty}\sqrt[k]{k|A_k|q^k}\le\sqrt{q} $$
	and hence $R\geq\dfrac{1}{\sqrt{q}}$. Denote $\alpha=\dfrac{1}{\sqrt{q}}$. Then $\alpha>\beta>0$ and 
	$$r=\dfrac{\beta}{\alpha}=\dfrac{1}{\sqrt{q}}\leq\dfrac{1}{\sqrt{2}}.$$
Further, let's $\widetilde{a}_k=k A_k q^k.$ Then using the estimate \eqref{Ak}, we have
	$$|\widetilde{a}_k|\le c_2 (\sqrt{q})^k,\ \ c_2=4b_1(1)=8.$$
	Thus, all the conditions of the theorem \ref{thGor} are satisfied. So, we prove the following 
	\begin{corollary} Let $g$ be a real-valued multiplicative function, such that for any irreducible polynomial $P$ and every integer $k\ge 1$ the equality $d_k=g(P^k)$ holds for some arbitrary sequence of reals $\left\{d_k\right\}_{k=1}^\infty$. Set $$f(x)=1+d_1 x+d_2 x^2+\ldots.$$ Further, let $0<d_1<1$ and for $k\ge 1$ the inequality $k|a_k|\le 1$ holds, where values $a_k$ was defined in \eqref{akk}. 
		Then for $N\ge n+1$ we have
	\begin{equation}
	T(N)=(-1)^N\binom{-d_1}{N}q^N\left(a(q^{-1})+\sum_{k=1}^n\dfrac{\binom{k-d_1}{k}}{\binom{N+d_1-1}{k}}\dfrac{q^{-k}}{k!}a^{(k)}(q^{-1})\right)+R_n,
	\label{Gorod}
	\end{equation}
	where $R_n\ll_n\dfrac{q^N}{N^{1-d_1}}\left(\dfrac{1}{\sqrt{q}N}\right)^{n+1},$ and $a(x)$ determine by infinity product $$a(x)=\prod\limits_{l=1}^{+\infty}\left\{f(x^l)(1-x^l)^{d_1}\right\}^{\pi_q(l)},$$
	which converges in the circle centred at origin with radius $r=\dfrac{1}{\sqrt{q}}$.
	\end{corollary}
	\par Let $R'_n$ be the respectively remainder term that obtained from Theorem \ref{thm2}, so $R'_n\ll_n\dfrac{q^N}{N^{1-d_1}}\dfrac{1}{q^{n+1}}$. Note that
	if conditions of theorem \ref{thm2} are holds then $R'_n=o(R_n)$ if $q\to\infty$ and $N\ge 190$ is fixed. Further, in case of $q\to\infty,\ N\to\infty$ we have
	$$R'_n\ll_n R_n\ \ \text{for}\  N^2\ll q,\ \ R_n\ll_n R'_n\  \text{for} \ \ q\ll N^2.$$
	Note that
	$$a(q^{-1})=\prod_{l=1}^{+\infty}\left(f_2(q^{-l})\right)^{\pi_q(l)}=1+\sum_{k=2}^{+\infty}B_k=1+O\left(\dfrac{1}{q}\right).$$
	From here we get that the formula \eqref{Gorod} gives an asymptotic of $T(N)$ in case of $q^N\to \infty$ in contrast to the Theorem \ref{thm2}, which requires that $q\to\infty.$ In particular, the formula \eqref{Gorod} gives asymptotic of $T(N)$ at the fixed $q$ and $N\to\infty.$
	\section*{Examples}
	More detailed computations show that
	\begin{itemize}
		\item $h_1=d_1$
		\item $h_2=d_2-\dfrac{d_1(d_1+1)}{2}$
		\item $h_3=d_3-d_1d_2+\dfrac{d_1(d_1^2-1)}{3}$
	\end{itemize}
	These values $h_k$ allow us to write out the first few values of $A_l(N)$.
\begin{itemize}
	\item $A_0(N)=\binom{d_1+N-1}{N};$ 
	\item $A_1(N)=\binom{d_1+N-3}{N-2}\left(d_2-\dfrac{d_1(d_1+1)}{2}\right);$
	\item $A_2(N)=\binom{d_1+N-4}{N-3}\left(d_3-d_1d_2+\dfrac{d_1(d_1^2-1)}{3}\right)+\binom{d_1+N-5}{N-4}\binom{1+d_2-\frac{d_1(d_1+1)}{2}}{2};$
\end{itemize}
these values $A_l(N)$ allow us to write out the first three terms in the asymptotics given by the theorems \ref{thm1} and \ref{thm2}. Namely, we have
\begin{multline*}
T(N)=\binom{d_1+N-1}{N}q^N\left(1+\dfrac{\binom{N}{2}}{\binom{d_1+N-1}{2}}\left(d_2-\dfrac{d_1(d_1+1)}{2}\right)q^{-1}+\right.\\
\left\{\left.\dfrac{\binom{N}{3}}{\binom{d_1+N-1}{3}}\left(d_3-d_1d_2+\dfrac{d_1(d_1^2-1)}{3}\right)+\dfrac{\binom{N}{4}}{\binom{d_1+N-1}{4}}\binom{1+d_2-\frac{d_1(d_1+1)}{2}}{2}\right\}q^{-2}\right)+R,
\end{multline*}
where $R\ll_N q^{N-3}$ (in theorem \ref{thm1}) and $|R|\le 9.3\cdot d_1\dfrac{q^{N-3}}{N^{1-d_1}}$ (in theorem \ref{thm2}).
		\par Let us establish sufficient conditions for the inequality \eqref{conda} of theorem \ref{thm2}.
	 Now we establish the recursion for the sequence $a_k$ defined in \eqref{akk}. Let $f$ be the function defined in \eqref{f}, so we have
	$$\dfrac{ d \ln f(t)}{dt}f(t)=\dfrac{d f(t)}{d t}.$$
Expanding both sides of this equality in power series of $t$ and comparing the coefficients of $t^n$, after simple computations we obtain
	$$a_{n+1}=\dfrac{1}{n+1}\left((n+1)d_{n+1}-\sum_{k=1}^n{ka_kd_{n+1-k}}\right),\ \ n\ge 0,$$
	in particular $a_1=d_1$.
 \begin{propA} Let $g$ be a multiplicative function such that for any irreducible polynomial $ P $ over $ \mathbb{F} _q $ and any $ k \ge 1 $ the following inequality $d_k=g(P^k)$ holds, where $\left\{d_k\right\}_{k=1}^{\infty}\subset \mathbb{R}$ is arbitrary fixed sequence. Assume that
 	\begin{itemize}
 		\item $d_1\in(0,1)$;
 		\item $d_{k+1}\le d_k $ for $k\ge 1$;
 		\item $ kd_k\le (k+1)d_{k+1}$  for $k\ge 1$;
 	\end{itemize}
 	then the sum $T(N)$ has the decompositions \eqref{TN}, and \eqref{Gorod}.
 \end{propA}
 \begin{proof}
 	It is sufficient to show that the following inequality
 	\begin{equation}
 	0<a_k<\dfrac{1}{k}
 	\label{ak2}
 	\end{equation}
 	holds. For $k=1$ we have $a_1=d_1\in(0,1).$\par Assume that \eqref{ak2} holds for $k=1, \ldots, n,\ \ n\ge 1.$ Now we prove that this inequality holds for $k=n+1$. By virtue of the conditions of the proposition and the assumption of induction we have
 	\begin{multline}a_{n+1}=\dfrac{1}{n+1}\left((n+1)d_{n+1}-\sum_{k=1}^n ka_k d_{n+1-k}\right)=\\
 	\dfrac{1}{n+1}\left((n+1)d_{n+1}-nd_n+nd_n-\sum_{k=1}^{n-1}ka_k d_{n-k}+\sum_{k=1}^{n-1}ka_k (d_{n-k}-d_{n+1-k})-na_n d_1\right)=\\
 	\dfrac{1}{n+1}\left((n+1)d_{n+1}-nd_n+(1-d_1)na_n+\sum_{k=1}^{n-1}ka_k (d_{n-k}-d_{n+1-k})\right)>0.
 	\end{multline}
 	On the other hand, by virtue of the assumption of induction
 	\begin{multline}
 	a_{n+1}=\dfrac{1}{n+1}\left((n+1)d_{n+1}-nd_n+(1-d_1)na_n+\sum_{k=1}^{n-1}ka_k (d_{n-k}-d_{n+1-k})\right)<\\
 	\dfrac{1}{n+1}\left(n(d_{n+1}-d_n)+d_{n+1}+(1-d_1)+\sum_{k=1}^{n-1} (d_{n-k}-d_{n+1-k})\right)=\\
 	\dfrac{1}{n+1}\left((n+1)(d_{n+1}-d_n)+1\right)\le\dfrac{1}{n+1}.
 	\end{multline}
 	Thus the proposition is proved.
 \end{proof}
	\begin{propB} Let $g$ be a multiplicative function such that for any irreducible polynomial $ P $ over $ \mathbb{F} _q $ and any $ k \ge 1 $ the following inequality $d_k=g(P^k)$ holds, where $\left\{d_k\right\}_{k=1}^{\infty}\subset \mathbb{R}$ is an arbitrary fixed sequence. Then if $d_1\in(0,1)$ and for all $n\ge 1$ there is the inequality
		\begin{equation}
		\sum_{k=1}^n|d_k|+(n+1)|d_{n+1}|\le 1,
		\label{d1dn}
		\end{equation}
		then the sum $T(N)$ has the decompositions \eqref{TN}, and \eqref{Gorod}.
	\end{propB}
	\begin{proof}
		\par It suffices to show that inequality \eqref{conda} follows from inequality \eqref{d1dn}. 
		\\For $n=1$ we have $a_1=d_1\in(0,1)$. Hence $|a_1|\le 1$. 
		\par Let us suppose that $|a_k|\le\dfrac{1}{k}$ for $1\le k\le n,\ \ n\ge1.$ Then
		\begin{multline}
		|a_{n+1}|\le\dfrac{1}{n+1}\left(|d_{n+1}|(n+1)+\sum_{k=1}^n|ka_k d_{n+1-k}|\right)\le\\
		\dfrac{1}{n+1}\left(|d_{n+1}|(n+1)+\sum_{k=1}^n|d_{n+1-k}|\right)\le\dfrac{1}{n+1}.
		\end{multline}
		Thus the inequality \eqref{conda} is proved. By induction the proposition is proved.
	\end{proof}
	
We give examples of functions $g$ that satisfy the conditions of proposition A:
	\begin{itemize}
		\item $g_1(F)=(\tau(F))^{-\alpha},\ 0<\alpha\le 1. $
		\item $g_2(F)=c^{\omega(F)},\ 0<c< 1 . $
		\item $g_3(F)=\dfrac{\tau_m(F)}{\tau_{m+1}(F)},\ m\ge 2. $
		\item $g_4(F)=\dfrac{1}{\tau(F^r)},\ r\ge 1$.
	\end{itemize}
	We also give examples of functions $g$ that satisfy the conditions of proposition B:
	\begin{itemize}
		\item $g_5(F)=(\tau(F))^{-\alpha},\ \alpha\ge 2. $
		\item $g_6(F)=\dfrac{1}{\tau_m(F)},\ m\ge 3. $
		\item $g_7(F)=\dfrac{1}{\left(\tau_{m_1}(F)\right)^{\gamma_1}\ldots\left(\tau_{m_k}(F)\right)^{\gamma_k}},\ \ k\ge2,\ \gamma_i\ge1$.
		\item $ g_8(F)=\dfrac{1}{\tau_m(F^r)},\ r\ge 1,\ m\ge 3$.
	\end{itemize}
	For the function $g_1$, the condition $k d_k\le(k+1)d_{k+1}$ is equivalent to the inequality $$\left(1+ \dfrac{1}{k+1}\right)^{\alpha}\le 1+\dfrac{1}{k}.$$
	The remaining conditions for $g_1$ are obviously hold.	For the functions $g_k,\ k\in\{2,3,4\}$, the conditions of proposition A are easily verified.
	\par For the function $g_5$ we have $d_k=g_5(P^k)= \dfrac{1}{(k+1)^\alpha}$. Hence for $\alpha\ge 2$ we have
	\begin{multline}
	\sum_{k=1}^n\dfrac{1}{(k+1)^\alpha}+\dfrac{n+1}{(n+2)^\alpha}\le\sum_{k=1}^n\dfrac{1}{(k+1)^2}+\dfrac{n+1}{(n+2)^2}=\\
	\dfrac{\pi^2}{6}-1-\sum_{k=2}^{+\infty}\dfrac{1}{(n+k)^2}+\dfrac{n+1}{(n+2)^2}\le \dfrac{\pi^2}{6}-1-\int_{n+2}^{+\infty}\dfrac{dt}{t^2}+\dfrac{n+1}{(n+2)^2}\le\dfrac{\pi^2}{6}-1<1.
	\end{multline}
	For the function $g_6$ we get $$d_k=g_6(P^k)=\dfrac{1}{\binom{m+k-1}{m-1}}\le\dfrac{2}{(k+1)(k+2)},\ m\ge 3.$$
	Hence
	$$\sum_{k=1}^n\dfrac{2}{(k+1)(k+2)}+\dfrac{2(n+1)}{(n+2)(n+3)}=1-\dfrac{4}{(n+2)(n+3)}<1.$$
	For the functions $g_7$ and $g_8$, proposition B is satisfied due to obvious inequalities
		$$g_7(P^k)\le g_5(P^k)\ \text{for $\alpha=2$},\ \ g_8(P^k)\le g_6(P^k).$$
	\par In what follows, $F$ runs through the set of monic polynomials. The following 2 examples refer to the theorem \ref{thm2}. For $q\to\infty$ and $N\ge710$ or for $q\to\infty$ and $N\to\infty$ we have
	\begin{multline*}
	\sum_{\deg F=N}\dfrac{1}{2^{\omega(F)}}=\dfrac{q^N}{4^N}\left(\binom{2N}{N}+2\binom{2N-4}{N-2}q^{-1}+\left(8\binom{2N-6}{N-3}+18\binom{2N-8}{N-4}\right)q^{-2}\right)+R_1;
	\end{multline*}
	$$
	\sum_{\deg F=N}\dfrac{1}{\tau(F)}=\dfrac{q^N}{4^N}\left(\binom{2N}{N}-\dfrac{2}{3}\binom{2N-4}{N-2}q^{-1}-\left(\dfrac{8}{3}\binom{2N-6}{N-3}+\dfrac{46}{9}\binom{2N-8}{N-4}\right)q^{-2}\right)+R_2;
	$$
	where $R_1,\,R_2\ll\frac{q^{N-3}}{\sqrt{N}}$ and implied costant in the symbol $\ll$ is (at most) 5.
	\par The following two examples refer to the corollary of the theorem \ref{thGor}.
	\\For $q^N\to\infty$ we have
	$$\sum_{\deg F=N}\dfrac{1}{2^{\omega(F)}}=C_0\cdot\dfrac{\binom{2N}{N}}{4^N}q^N+O\left(\dfrac{q^{N-0.5}}{N^{1.5}}\right),\ \ C_0=\prod_{l=1}^{+\infty}\left(\dfrac{2q^l-1}{2\sqrt{q^{2l}-q^{l}}}\right)^{\pi_q(l)};$$
	$$\sum_{\deg F=N}\dfrac{1}{\tau(F)}=C_1\cdot\dfrac{\binom{2N}{N}}{4^N}q^N+O\left(\dfrac{q^{N-0.5}}{N^{1.5}}\right),\ \ C_1=\prod_{l=1}^{+\infty}\left(\sqrt{q^{2l}-q^{l}}\ln\dfrac{q^l}{q^l-1}\right)^{\pi_q(l)}.$$
	The following 4 examples refer to the theorem \ref{thm1}.
	$$\sum_{\deg F=N}\tau_k(F^r)=\binom{\binom{k+r-1}{k-1}+N-1}{N}q^N+O_{N,k,r}\left(q^{N-1}\right),\ \ N\ge 1,\  q\to\infty;
$$
$$\sum_{\deg F=3}\tau_3(F^2)=56q^3-36q^2+8q;
$$
$$
\sum_{\deg F=3}\dfrac{1}{2^{\omega(F)}}=\dfrac{5q^3+q^2+2q}{16};
$$
$$
\sum_{\deg F=3}\dfrac{1}{\tau(F)}=\dfrac{15q^3-q^2-2q}{48};
$$
	

	\end{document}